\documentclass[12pt]{amsart}
\usepackage{amssymb}
\usepackage{hyperref}

\theoremstyle{plain}
\newtheorem{theorem}{Theorem}
\newtheorem{lemma}{Lemma}
\newtheorem{corollary}{Corollary}

\newtheorem{definition}{Definition}
\newtheorem{question}{Question}
\newtheorem{example}{Example}

\title{High degree $b$-Niven numbers}

\author{Viorel Ni\c{t}ic\u{a}}
\address{Department of Mathematics\\ West Chester
University of Pennsylvania\\ West Chester, PA 19383, USA\\
vnitica@wcupa.edu}

\begin{document}

\begin{abstract} Let $b$ be a numeration base. A $b$-Niven number is one that is divisible by the sum of its base $b$ digits.
We introduce high degree $b$-Niven numbers. These are $b$-Niven numbers that have a power greater than $1$ that is $b$-Niven number. 
Our main result shows that for each degree there exists an infinite set of bases $b$ for which $b$-Niven numbers of that degree exist.
The high degree $b$-Niven numbers are given by explicit formulas and have all digits different from zero.
\end{abstract}

\maketitle

\section{Introduction}\label{sec:1}

Niven (or Harshad) numbers are numbers divisible by the sum of their decimal digits. Niven numbers have been extensively studied. See for instance Cai \cite{C}, Cooper and Kennedy \cite{CK}, De Koninck and Doyon \cite{KD}, and Grundman \cite{G}. Of interest are also $b$-Niven numbers, which are numbers divisible by the sum of their base $b$ digits. See for example Fredricksen, Iona\c scu, Luca, and St\u anic\u a \cite{FILS}. Some variants of Niven numbers can be found in Boscaro \cite{B1} and Bloem \cite{B2}.

In Ni\c tic\u a \cite{N} we introduce degree $2$ $b$-Niven numbers. These are $b$-Niven numbers for which the squares are $b$-Niven numbers. The goal of this paper is to introduce high degree $b$-Niven numbers and to show that for each degree there exists an infinite set of bases in which $b$-Niven numbers of that degree appear. We observe that the high degree $b$-Niven numbers shown in this paper are given by explicit formulas and have all digits different from zero. This distinction is important if one wants to avoid the trivial examples of $b$-Niven numbers given by the powers of $10_b$. 

In what follows let $b\ge 2$ be a numeration base. We denote by $s_b(N)$ the sum of the base $b$ digits of the integer $N$. If $x$ is a string of digits, we denote by $(x)^{\land k}$ the base $10$ integer obtained by repeating $x$ $k$-times. We denote by $[x]_b$ the value of the string in base $b$.

\begin{definition} Let $m\ge 1$ be an integer. An integer $N$ is called a \emph{degree $m$ $b$-Niven number} if $N$ and $N^m$ are $b$-Niven numbers. 
\end{definition}

It is shown in \cite{N} that if $b\equiv 3\pmod{4}$ then there exists an infinite set of degree $2$ $b$-Niven numbers. It is obvious that any power of $10_b$ is a degree $m$ $b$-Niven number for any $m$ and for any $b$.  Moreover, if $N$ is a degree $m$ $b$-Niven number then $10_b\times N$ is also a  degree $m$ $b$-Niven number. 

The following theorem gives a criterion for the existence of even degree $b$-Niven numbers with all digits different from zero.

\begin{theorem}\label{thm:1} Let $b\ge 2, k\ge 1,n\ge 1$ be integers. Consider the number:
\begin{equation*}
N_k=[b^{2^{k}}-1]_b.
\end{equation*}

Assume that $n=2^qp, p\ge 0$ odd, $q\ge 0$, $b\ge C_{2n}^n$, and $b=(4\ell+2)p+1, \ell \ge 0$. Then $N_k$ is a degree $2n$ $b$-Niven number.

In particular, for any even degree, there exists an infinite set of bases in which $b$-Niven numbers of that degree, with all digits different from zero, exist.
\end{theorem}

The proof of Theorem \ref{thm:1} is done in Section \ref{sec:2}. 

The following lemma gives estimates for the size of $b$ in Theorem \ref{thm:1}. 

\begin{lemma} If $n\ge 1$ is a positive integer, then:
\begin{equation*}
\frac{4^n}{\sqrt{n}}\le C_{2n}^n\le \frac{4^n}{\root 3 \of n}.
\end{equation*}
\end{lemma}

The following theorem gives a criterion for the existence of odd degree $b$-Niven numbers with all digits different from zero.

\begin{theorem}\label{thm:2} Let $b\ge 2, k\ge 1,n \ge 1$ be integers. Consider the numbers:
\begin{equation*}
N_k=[b^{2^{k}}-1]_b.
\end{equation*}

Assume that $\frac{n+1}{2}=2^qp, p\ge 0$ odd, $q\ge 0$, $b\ge C_{n}^{\frac{n+1}{2}}$, and $b=(4\ell+2)p+1, \ell\ge 0$. Then $N_k$ is a degree $n$ $b$-Niven number.

In particular, for any odd degree, there exists an infinite set of bases in which $b$-Niven numbers of that degree, with all digits different from zero, exist. 
\end{theorem}

The proof of Theorem \ref{thm:2} is done in Section \ref{sec:3}.

Theorems \ref{thm:1} and \ref{thm:2} have the following corollary.

\begin{corollary} For any degree there exists an infinite set of bases in which $b$-Niven numbers of that degree, with all digits different from zero, exist. 
\end{corollary}

We do not know how to answer the following question.

\begin{question} Does there exist a base $b$ for which $b$-Niven numbers of arbitrary high degree, with all digits different from zero, exist?
\end{question}

The following theorem shows an infinite set of bases containing $b$-Niven numbers of multiple consecutive high degrees, with all digits different from zero.

\begin{theorem}\label{thm:3} Let $b\ge 2, d\ge 1$. Consider the sequence $(N_k)_{k\ge 1}$, where $N_k=[b^{2^k}-1]_b$. Then there exists an infinite set of bases $b$ in which the sequence $(N_k)_{k\ge 1}$ consists of $b$-Niven numbers of degree $i$ for all $1\le i\le d$. 
\end{theorem}

\begin{proof} The assumptions about $b$ needed to apply Theorems \ref{thm:1} and \ref{thm:2} for $1\le i\le d$  are as follows:
\begin{enumerate}
\item $b\ge C_{2i}^i$ if $i$ is even;
\item $b\ge C_{i}^{\frac{i+1}{2}}$ if $i$ is odd;
\item $b\equiv 3\pmod{4}$;
\item $b\equiv 1\pmod{p_i}$, where $i=2^{\ell_i}p_i$ with $p_i$ odd.
\end{enumerate}

The first two conditions are valid for $b\ge C_{2d}^d$, if $d$ is even, and for $b\ge C_{d}^{\frac{d+1}{2}}$, if $d$ is odd. The system of congruences has an infinite set of solutions given by:
\begin{equation*}
b=(4\ell+2)p_1p_2\cdots p_d+1, \ell \ge 1.
\end{equation*}
\end{proof}

\begin{example}\label{ex:some} Assume $b\equiv 3\pmod{4}$ and consider the sequence $(N_k)_{k\ge 1}$, where $N_k=[b^{2^k}-1]_b$. The numbers $N_k$ are degree $2$ $b$-Niven numbers and degree $3$ $b$-Niven numbers. If $b\ge 6$, then $N_k$ also are degree $4$ $b$-Niven numbers.
If  $b\ge 10$ and $b\equiv 1\pmod{3}$, then $N_k$ also are degree $5$ $b$-Niven numbers. If $b\ge 20$ and $b\equiv 1\pmod{3}$, then $N_k$ also are degree $6$ $b$-Niven numbers. If $b\ge 35$, then $N_k$ also are degree $7$ $b$-Niven numbers. If $b\ge 70$, then $N_k$ also are degree $8$ $b$-Niven numbers. If  $b\ge 126$ and $b\equiv 1\pmod{5}$ then $N_k$ also are degree $9$ $b$-Niven numbers. If $b\ge 252$ and $b\equiv 1\pmod{5}$ then $N_k$ also are degree $10$ $b$-Niven numbers. If $b\ge 66$ and $b\equiv 1\pmod{3}$, then $N_k$ also are degree $11$ $b$-Niven numbers. If  $b\ge 924$ and $b\equiv 1\pmod{3}$, then $N_k$ also are degree $12$ $b$-Niven numbers. If  $b\ge 1716$ and $b\equiv 1\pmod{7}$, then $N_k$ also are degree $13$ $b$-Niven numbers. If  $b\ge 3432$ and $b\equiv 1\pmod{7}$, then $N_k$ also are degree $14$ $b$-Niven numbers. If $b\ge 6435$, then $N_k$ also are degree $15$ $b$-Niven numbers.

Summing up we conclude that if $b\ge 6435$ and $b=105(4\ell+2)+1, \ell \ge 1$, then the numbers $N_k$ are degree $i$ $b$-Niven numbers for all $1\le i\le 15$. The smallest base $b$ that satisfies above conditions is $b=6511$. The base $b$ representations for $N_k$ and its powers can easily be derived using formulas \eqref{eq:contr} and \eqref{eq:contr1} from the proofs of Theorem \ref{thm:1} and \ref{thm:2}. For example:
\begin{itemize}
\item $N_1=[(6510)(6510)]_{6511}$
\item $N_1^2=[(6510)(6509)(0)(1)]_{6511}$
\item $N_1^3=[(6510)(6508)(0)(2)(6510)(6510)]_{6511}$
\item $N_1^4=[(6510)(6507)(0)(5)(6510)(6507)(0)(1)]_{6511}$
\item $N_1^5=[(6510)(6506)(0)(9)(6501)(6501)(0)(4)(6510)(6510)]_{6511}$.
\end{itemize}
\end{example}

\begin{question} Due to the condition $b\equiv 3\pmod{4}$, the bases claimed in Theorem \ref{thm:3} are all odd. Is the result in Theorem \ref{thm:3} true for an infinite set of even bases?
\end{question}

\section{Proof of Theorem \ref{thm:1}}\label{sec:2}

Note that the set $\{b\vert b=(4\ell+2)p+1, \ell \ge 0\}$ is the solution of the system of equations $b\equiv 3\pmod{4}, b\equiv 1\pmod{p}$.

An equivalent representation for $N_k$ is $[(b-1)^{\land 2^{k}}]_b$, so $s_b(N_k)=(b-1)2^{k}$. As $b$ is odd, $\gcd(b,2)=1$. As $\phi(2^{k+1})=2^k$ Euler's theorem gives that $2^{k+1}$ divides $N_k$. In addition $b-1$ divides $N_k$. As $b\equiv 3\pmod{4}$, $\gcd(b-1,2)=2$. It follows that $(b-1)2^{k}$ divides $N_k$, so $N_k$ is a $b$-Niven number. 

We compute $s_b(N_k^{2n})$ using the assumption $b\ge C_{2n}^n$. From binomial formula follows that:
\begin{equation}\label{eq:sum1}
\begin{gathered}
N_k^{2n}=(b^{2^{k}}-1)^{2n}=\sum_{i=1}^{n}\Big ( C_{2n}^{2i}b^{2i\cdot 2^{k}}-C_{2n}^{2i-1}b^{(2i-1)\cdot 2^{k}}\Big ) +1.
\end{gathered}
\end{equation}

Each one of the difference $C_{2n}^{2i}b^{2i\cdot 2^{k}}-C_{2n}^{2i-1}b^{(2i-1)\cdot 2^{k}}$ in \eqref{eq:sum1} has a base $b$ representation given by:
\begin{equation}\label{eq:contr}
[(C_{2n}^{2i}-1)(b-1)^{\land 2^{k}-1}(b-C_{2n}^{2i-1})(0)^{\land (2i-1)\cdot 2^{k}-1}]_b.
\end{equation}

As the contributions \eqref{eq:contr} to the base $b$ representation of $N_k^{2n}$ do not overlap one has that:
\begin{equation}\label{eq:contr1}
\begin{aligned}
s_b(N_k^{2n})&=n(b-1)(2^{k}-1)\\
&+(b-C_{2n}^1)+(b-C_{2n}^3)+\cdots + (b-C_{2n}^{2n-1})\\
&+(C_{2n}^2-1)+(C_{2n}^4-1)+\cdots +(C_{2n}^{2n-2}-1)\\
&+1,
\end{aligned}
\end{equation}
where second line in \eqref{eq:contr1} has $n$ terms and third line has $n-1$ terms.

If one rearranges the terms in \eqref{eq:contr1} and uses that:
\begin{equation*}
\begin{gathered}
C_{2n}^1+C_{2n}^3+\cdots +C_{2n}^{2n-1}=2^{2n-1}\\
C_{2n}^2+C_{2n}^4+\cdots +C_{2n}^{2n-2}=2^{2n-1}-2,
\end{gathered}
\end{equation*}
formula \eqref{eq:contr1} becomes:
\begin{equation*}
\begin{aligned}
s_b(N_k^{2n})&=n(b-1)(2^{k}-1)+nb-2^{2n-1}+2^{2n-1}-2-(n-1)+1\\
&=n(b-1)2^{k}.
\end{aligned}
\end{equation*}

It remains to show that $n(b-1)2^{k}$ divides $N_k^{2n}$. From above it follows that $(b-1)2^{k}$ divides $N_k$. To finish the proof we show that $n$ divides $N_k^{2n-1}$. 

The assumption $b\equiv 3\pmod{4}$ implies that $2$ divides $b-1$ and the assumption $b\equiv 1\pmod{p}$ implies that $p$ divides $b-1$. 
As $\gcd(2,p)=1$, $2p$ divides $b-1$. But $b-1$ divides $N_k$, so 
\begin{equation}\label{eq1}
2^{2n-1}p\vert N_k^{2n-1}.
\end{equation} 
Due to the fact, easily proved by induction, that $2^{\ell}$ is a factor of $2^{2^{\ell} -1}$ for $\ell \ge 1$, one has that
\begin{equation}\label{eq2}
2^q\vert 2^{2^q-1}\vert 2^{2^q-1}\cdot 2^{2^{q+1}p-2^q}=2^{2^{q+1}p-1}=2^{2n-1}.
\end{equation}

Combining \eqref{eq1} and \eqref{eq2} and using that $\gcd(p, 2^q)=1$ one has that $n=2^qp$ divides $N_k^{2n-1}$.
 
\section{Proof of Theorem \ref{thm:2}}\label{sec:3} 

Note that the set $\{b\vert b=(4\ell+2)p+1, \ell \ge 0\}$ is the solution of the system of equations $b\equiv 3\pmod{4}, b\equiv 1\pmod{p}$.

It follows from the first paragraph of Section \ref{sec:2} that $N_k$ is a $b$-Niven number. 

We compute $s_b(N_k^{n})$ using the assumption that $b\ge C_{n}^\frac{n+1}{2}$. From binomial formula follows that:
\begin{equation}\label{eq:sum1*}
\begin{gathered}
N_k^{n}=(b^{2^{k}}-1)^{n}=\sum_{i=1}^{\frac{n+1}{2}}\Big ( C_{n}^{n+2-2i}b^{(n+2-2i)\cdot 2^{k}}-C_{n}^{n+1-2i}b^{(n+1-2i)\cdot 2^{k}}\Big ).
\end{gathered}
\end{equation}

Each one of the difference $C_{n}^{n+2-2i}b^{(n+2-2i)\cdot 2^{k}}-C_{n}^{n+1-2i}b^{(n+1-2i)\cdot 2^{k}}$ in \eqref{eq:sum1*} has a base $b$ representation given by:
\begin{equation}\label{eq:contr*}
[(C^{n+2-2i}_{n}-1)(b-1)^{\land 2^{k}-1}(b-C_{n}^{n+1-2i})(0)^{\land (n+1-2i)\cdot 2^{k}-1}]_b.
\end{equation}

As the contributions \eqref{eq:contr*} to the base $b$ representation of $N_k^{n}$ do not overlap one has that:
\begin{equation}\label{eq:contr1*}
\begin{aligned}
s_b(N_k^{n})&=\frac{n+1}{2}(b-1)(2^{k}-1)\\
&+(b-C_{n}^{n-1})+(b-C_{n}^{n-3})+\cdots + (b-C_{n}^{0})\\
&+(C_{n}^n-1)+(C_{n}^{n-2}-1)+\cdots +(C_{n}^{1}-1),
\end{aligned}
\end{equation}
where second and third lines in \eqref{eq:contr1*} have $\frac{n+1}{2}$ terms.

If one rearranges the terms in \eqref{eq:contr1*} and uses that:
\begin{equation}
\begin{gathered}
C_{n}^0+C_{n}^2+\cdots +C_{n}^{n-1}=2^{n-1}-1\\
C_{n}^1+C_{n}^3+\cdots +C_{n}^{n}=2^{n-1}-1,
\end{gathered}
\end{equation}
formula \eqref{eq:contr1*} becomes:
\begin{equation}
\begin{aligned}
s_b(N_k^{n})&=\frac{n+1}{2}(b-1)(2^{k}-1)+\frac{n+1}{2}b-2^{n-1}+2^{n-1}-\frac{n+1}{2}\\
&=\frac{n+1}{2}(b-1)2^{k}.
\end{aligned}
\end{equation}

It remains to show that $\frac{n+1}{2}(b-1)2^{k}$ divides $N_k^{n}$. From above it follows that $(b-1)2^{k}$ divides $N_k$. To finish the proof of the theorem we show that $\frac{n+1}{2}$ divides $N_k^{n-1}$. 

The assumption $b\equiv 3\pmod{4}$ implies that $2$ divides $b-1$ and the assumption $b\equiv 1\pmod{p}$ implies that $p$ divides $b-1$. 
Overall, as $\gcd(2,p)=1$, $2p$ divides $b-1$. As $b-1$ divide $N_k$ we conclude that:
\begin{equation}\label{eq3}
2^{n-1}p\vert N_k^{n-1}.
\end{equation} 
Due to the fact, easily proved by induction, one has that:
\begin{equation}\label{eq4}
2^q\vert 2^{2^q-1}\vert 2^{2^q-1}\cdot 2^{2^q\cdot p-2^q}=2^{2^q\cdot p-1}=2^{n-1}.
\end{equation}

Combining \eqref{eq3} and \eqref{eq4}, and using that $\gcd(2^q,p)=1$ one has that  $\frac{n+1}{2}=2^qp$ divides $N_k^{n-1}$.

\end{document}